\documentclass[10pt]{amsart}
\usepackage{amsmath}
\usepackage{amsfonts}
\usepackage{amssymb}
\usepackage{amsthm}
\usepackage{url}
\usepackage{dsfont}
\usepackage{graphicx}
\usepackage{caption}
\usepackage{subcaption}
\usepackage{comment} 
\usepackage{stmaryrd}
\usepackage{hyperref}
\usepackage{todonotes}
\usepackage{color}
\usepackage{enumerate}

\numberwithin{equation}{section}

\newtheorem{prop}{Proposition}[section]
\newtheorem{lem}[prop]{Lemma}

\newtheorem{thm}[prop]{Theorem}

\newtheorem{cor}[prop]{Corollary}

\newtheorem{proposition}{Proposition}[section]

\theoremstyle{definition}
\newtheorem{definition}[proposition]{Definition}

\DeclareMathOperator{\Bl}{Bl}
\DeclareMathOperator{\End}{End}
\DeclareMathOperator{\Id}{Id}
\DeclareMathOperator{\tr}{tr}
\DeclareMathOperator{\Euc}{Euc}
\DeclareMathOperator{\Image}{Im}
\DeclareMathOperator{\GL}{GL}

\newcommand{\C}{\mathbb{C}}

\newcommand{\G}{\mathcal{G}}

\renewcommand{\epsilon}{\varepsilon}
\newcommand{\scH}{\mathcal{H}}

\title{Hermitian Yang-Mills connections on blowups}
\title[Hermitian Yang-Mills connections on blowups]{Hermitian Yang-Mills connections on blowups}

\author[Ruadha\'i Dervan and Lars Martin Sektnan]{Ruadha\'i Dervan and Lars Martin Sektnan}
\address{Ruadha\'i Dervan, DPMMS, Centre for Mathematical Sciences, Wilberforce Road, Cambridge CB3 0WB, United Kingdom}
\email{R.Dervan@dpmms.cam.ac.uk}

\address{Lars Martin Sektnan, D\'epartement de math\'ematiques, Universit\'e du Qu\'ebec \`a Montr\'eal, Case postale 8888, succursale
centre-ville, Montr\'eal (Qu\'ebec), H3C 3P8, Canada}
\email{lars.sektnan@cirget.ca}

\begin{document}

\begin{abstract}  Consider a vector bundle over a K\"ahler manifold which admits a Hermitian Yang-Mills connection. We show that the pullback bundle on the blowup of the K\"ahler manifold at a collection of points also admits a Hermitian Yang-Mills connection, for K\"ahler classes on the blowup which make the exceptional divisors small. Our proof uses gluing techniques, and is hence asymptotically explicit. This recovers, through the Hitchin-Kobayashi correspondence, algebro-geometric results due to Buchdahl and Sibley. \end{abstract}

\maketitle

\section{Introduction}

A basic problem in K\"ahler geometry is to determine which Hermitian vector bundles admit Hermitian Yang-Mills connections. Initially motivated by physics, these connections play an important role in the moduli theory of vector bundles, as they give a canonical choice of connection. The goal of the present work is more concrete: we give a new, reasonably explicit, construction of Hermitian Yang-Mills connections. 

We begin with a compact $n$-dimensional K\"ahler manifold $X$ with K\"ahler metric $\omega$, and a Hermitian vector bundle $(E,h)$ on $X$. Let $p\in X$ be a point, and denote by $\pi : \Bl_p X \rightarrow X$ the blowup of $X$ at $p$ with exceptional divisor $F$. For each $\epsilon>0$ sufficiently small, the class $[\pi^*\omega] - \epsilon c_1(F)$ is a K\"ahler class and we shall define an explicit K\"ahler metric $\omega_{\epsilon} \in [\pi^*\omega] - \epsilon c_1(F),$ which away from a neighbourhood of $F$ is simply $\pi^*\omega$. Similarly we shall define an explicit Hermitian metric $h_{\epsilon}$ on $\pi^*E$ which is simply $\pi^*h$ away from a neighbourhood of $F$. Our main result is as follows:

\begin{thm}\label{blthm} Suppose $(E,h)$ admits a Hermitian Yang-Mills connection. Then for all $\epsilon>0$ sufficiently small, the pullback $(\pi^*E,h_{\epsilon})$ admits a Hermitian Yang-Mills connection with respect to $\omega_{\epsilon}$.
\end{thm}

The Hermitian Yang-Mills connection on $E$ induces a holomorphic structure, which in turn induces a holomorphic structure on $\pi^*E$. The holomorphic structure induced from the Hermitian Yang-Mills connection produced by the result above is biholomorphic to this holomorphic structure. 

The proof of this result uses gluing techniques which go back at least to the work of Taubes \cite{CT}, and which have been applied to problems in K\"ahler geometry by Arezzo-Pacard \cite{arezzopacard06}, Arezzo-Pacard-Singer \cite{arezzopacardsinger11}, Sz\'ekelyhidi \cite{szekelyhidi12} and others. 

The existence of a Hermitian Yang-Mills connection on $\pi^*E$ is in fact independent of choice of K\"ahler metric within a fixed K\"ahler class. Thus Theorem \ref{blthm} proves the existence of a Hermitian Yang-Mills with respect to any K\"ahler metric within the class $[\pi^*\omega] - \epsilon c_1(F)$. What we wish to emphasise here is that our construction is essentially explicit in the limit $\epsilon\to 0$ for our choice of $h_{\epsilon}$ and $\omega_{\epsilon}$.

\begin{thm} Consider the Burns-Simanca metric $\omega_{BS}$ on $\Bl_{0}\C^n$, together with the pullback of the trivial vector bundle on $\C^n$ with the flat Hermitian metric. Set $\omega_{\epsilon}$ to be a gluing of $\omega$, away from a neighbourhood of $p$, and $\omega_{BS}$ on a ball around $F$. Similarly let $h_{\epsilon}$ be a gluing of $h$ to the flat metric. Then with respect to these metrics and for all $\epsilon>0$ sufficiently small, the Hermitian Yang-Mills connection on $\pi^*E$ is a small perturbation of order $\epsilon$ of the pullback of the Hermitian Yang-Mills connection on $E$.
\end{thm}

The key property of the Burns-Simanca metric on $\Bl_{0}\C^n$ that we use is that it is asymptotically flat, see for example \cite[Section 8.1.2]{szekelyhidi14book}. Moreover, it is defined completely explicitly. For applications to physics, it is of great importance to have explicit solutions of the Hermitian Yang-Mills equation; we refer to \cite{ABKO, DKLR} for a sample of recent work in this direction. Thus starting with an explicit Hermitian Yang-Mills connection on $E$, our construction gives an explicit construction of such connections of the pullback of $E$ to $\Bl_p X$ up to terms of order $\epsilon$. In particular our construction gives asymptotically explicit Hermitian Yang-Mills connections on infinitely many K\"ahler manifolds of distinct topological type. 

A small variant of our construction gives the following:

\begin{thm}\label{many-points}Denote by $\pi : \Bl_{p_1, \cdots, p_m} X \rightarrow X$ the blowup of $X$ at a collection of points $p_1, \cdots, p_m \in X$, and let $F_1,\hdots,F_m$ be the corresponding exceptional divisors. Then for any constants $a_1, \cdots, a_m > 0$ there exists a $\varepsilon_0 > 0$ such that for all $\varepsilon \in (0, \varepsilon_0)$ the pullback $\pi^*E$ admits a Hermitian Yang-Mills connection with respect to any K\"ahler metric $\omega_{\varepsilon}$ in the class $\pi^* [\omega] - \varepsilon \left( a_1 [F_1] + \cdots + a_k [F_m] \right)$.\end{thm}

We leave the details to the interested reader; the only subtlety is that one must blow up all points simultaneously rather than simply iterating Theorem \ref{blthm}. The reason is that iterating one would obtain Hermitian-Einstein metrics only with respect to classes of the form $$\pi^* [\omega] - \left( \varepsilon_1  a_1 [F_1] + \cdots + \varepsilon_k a_k [F_m] \right),$$ where the maximal value of $\varepsilon_i$ depends on the particular value of $\varepsilon_1, \cdots, \varepsilon_{i-1}.$ Theorem 1.3 says that we can avoid this dependence, and produces Hermitian-Einstein metrics with respect to K\"ahler classes in any direction into the K\"ahler cone beginning from $\pi^* [\omega]$. With appropriate choices of Hermitian and K\"ahler metrics, the Hermitian Yang-Mills connections produced by Corollary \ref{many-points} can again be made explicit up to terms of order $\epsilon$. 

The renowned Hitchin-Kobayashi correspondence, proved by Donaldson \cite{SD-surfaces, Donaldson-HK} and Uhlenbeck-Yau \cite{UY}, states that the existence of a Hermitian Yang-Mills connection on $E$ is equivalent to slope polystability of the corresponding holomorphic structure on $E$. Thus Theorem \ref{blthm} recovers the following algebro-geometric result, which is a classical result of Buchdahl in the case of surfaces \cite{NB}, and a recent result of Sibley in arbitrary dimensions \cite{BS}:

\begin{cor} Suppose $E\to(X,[\omega])$ is a stable holomorphic vector bundle over a K\"ahler manifold. Then the pullback bundle $\pi^*E$ is stable on $(Bl_pX,[\omega]-\epsilon F)$ for all $\epsilon>0$ sufficiently small.\end{cor}

Conversely, if $\pi^*E$ is stable with respect to $[\omega_{\epsilon}]$, it is straightforward to see that $E$ is itself stable with respect to $[\omega]$ \cite{NB}. It would be interesting to understand this result from the point of view of Hermitian Yang-Mills connections.

It is worth emphasising again that, while the Hitchin-Kobayashi correspondence abstractly produces a Hermitian Yang-Mills connection, it provides no information whatsoever as to what the connection looks like. By contrast, our technique gives very precise information. Buchdahl used the algebro-geometric result above to give a reasonably explicit study of the moduli space of stable vector bundles in certain situations \cite{NB}; it seems possible that Theorem \ref{blthm} could be used to give an even more detailed study of such moduli spaces.

A brief outline of the strategy of Theorem \ref{blthm} is as follows. From our gluing procedure, we obtain an approximate solution $A_{\epsilon}$ which is simply the Chern connection of $h_{\epsilon}$ with respect to the pullback holomorphic structure. The aim is to perturb this connection by an element of the complex gauge group in order to obtain a genuine solution. To achieve this, we use the contraction mapping theorem. For this to apply, the heart of the matter is to obtain appropriate bounds on the inverse of the linearisation of the Hermitian Yang-Mills operator in certain weighted H\"older spaces which we construct.

Finally, we end with a discussion of some analogues and extensions of Theorem \ref{blthm} that would be worth studying. First of all, one could study the corresponding problem for Hermitian Yang-Mills connections which are singular along a divisor $D\subset X$, which corresponds to parabolic stability of the vector bundle. The analogous problem for constant scalar curvature K\"ahler metrics with prescribed singularities along a divisor was studied in the PhD thesis of the second author \cite{sektnanthesis}. Secondly, it seems likely that there is an analogue of Theorem \ref{blthm} for Higgs bundles. Thirdly, a much studied problem in K\"ahler geometry has been the existence of extremal K\"ahler metrics on blowups \cite{arezzopacardsinger11,szekelyhidi12}; the analogy for vector bundles is direct sums of slope stable vector bundles of different slopes. Applying Theorem \ref{blthm} to each of the simple components will provide canonical connections on blowups, so our result applies verbatim in this setting. Fourthly, we expect an analogue of Theorem \ref{blthm} to hold for arbitrary Hermitian manifolds, which are not necessarily K\"ahler. While our proof suggests a strategy for proving this, some of our key estimates break down if $X$ is not K\"ahler. Lastly, we expect that an analogue of Theorem \ref{blthm} holds for blowups along arbitrary complex submanifolds, due to the recent work of Seyyadali-Sz\'ekelyhidi on extremal K\"ahler metrics on such blowups \cite{SS}. 

\vspace{4mm}

\noindent {\bf Outline:} We begin in Section \ref{prelim-HYM} with some standard definitions and results from the study of Hermitian Yang-Mills connections. Next in Section \ref{weighted-section} we define weighted Sobolev and H\"older spaces of sections of vector bundles, and establish some of their basic properties. We use these weighted norms in Section \ref{linear-section} to obtain the estimates on the inverse of the linearisation of the Hermitian Yang-Mills operator. We put the pieces together in Section \ref{main-section} to prove Theorem \ref{blthm} by a contraction mapping argument.

\vspace{4mm} \noindent {\bf Acknowledgements:} The authors would like to thank Ben Sibley for helpful discussions. The second named author is thankful to the CIRGET who support his postdoctoral position. 

 \section{Preliminaries on Hermitian Yang-Mills connections}\label{prelim-HYM} We refer to \cite{SK} for an introduction to Hermitian Yang-Mills theory. Let $X$ be a compact K\"ahler manifold of dimension $n$ with K\"ahler metric $\omega$. Fix a complex vector bundle $E$ of rank $m$ on $X$ together with a Hermitian metric $h$ on $E$. 

\begin{definition} We say that a connection $A$ which is compatible with $h$ is a \emph{Hermitian Yang-Mills connection} if
\begin{align*} F_A^{0,2} &= 0, \\
i \Lambda_{\omega} F_A &= c \cdot \Id_E,
\end{align*} where $\Lambda_{\omega} : \Omega^2 (\End E) \rightarrow \Omega^0 (\End E) $ is the adjoint of the Lefschetz map, $\Id_E\in \Omega^0(\End E)$ is the identity section, and $c$ is the topological constant defined by $$c = n\frac{\int_X c_1(E).[\omega]^{n-1}}{m\int_X [\omega]^{n}}.$$\end{definition}

The first condition implies that $A$ is \emph{integrable}, i.e. defines a holomorphic structure on $E$. The covariant derivative splits as $d_A = \partial_A + \bar{\partial}_A$, with integrability ensuring that $\bar{\partial}_A$ is induced from a holomorphic structure on the vector bundle. Hermitian Yang-Mills connections are unique, if they exist, and hence form a canonical choice of connection on a Hermitian vector bundle \cite[Proposition 2.2.2]{lubketeleman95}.

An equivalent point of view is to fix a \emph{holomorphic} vector bundle $E$ and search instead for a canonical Hermitian metric. 

\begin{definition}We say a Hermitian metric $h$ is \textit{Hermitian-Einstein} if  
\begin{align}\label{HEeqn} i \Lambda_{\omega} F_h = c \cdot \Id_E
\end{align}
for some constant $c$, where $F_h$ is the curvature of the Chern connection induced from the holomorphic structure on $E$. 

\end{definition} So one can either fix a Hermitian vector bundle and search for a canonical connection, or one can fix a holomorphic vector bundle, and search for a canonical Hermitian metric, which then induces the Chern connection. On occasion it will be useful to take this perspective instead.

The \textit{complex gauge group}, denoted $\mathcal{G}^{\mathbb{C}}$, is the group $\Gamma(\GL(E, \mathbb{C}))$. Fixing a Hermitian metric $h$ on $E$, for $f\in \G^{\C}$ it acts on the space of connections by
\begin{align}\label{connectionaction} d_{A^f} = f^* \circ \partial_A \circ (f^*)^{-1} + f^{-1} \circ \bar{\partial}_A \circ f ,
\end{align}
where $f^*$ denotes the adjoint of $f$ computed with respect to $h$. This action is such that $A^f$ is the Chern connection with respect to the holomorphic structure $f^{-1} \circ \bar{\partial}_A \circ f$. We say that connections $A,\hat{A}$ are \emph{gauge equivalent} if there exists an $f \in \mathcal{G}^{\mathbb{C}}$ such that $A^f=\hat{A}$. Remark that gauge equivalent connections define biholomorphic holomorphic structures. 

An object of central importance in the present work will be the linearisation of the Hermitian Yang-Mills operator. To calculate this, we first need some more additional notation. As above, we let A be an integrable connection, compatible with $h$, on the vector bundle $E$. Set \begin{equation}\label{linearisation-space} \scH= \mathcal{G}^{\mathbb{C}} \cap \Gamma \left( \End_{H} ( E, h) \right),\end{equation} where $\End_{H} (E, h)$ denotes the space of Hermitian endomorphisms of $(E,h)$. We will set $\mathcal{Q}$ to be the tangent space of $\scH$ at $\Id_E$, so that $\mathcal{Q}$ is the vector space consisting of Hermitian endormorphisms of $E$. Define the operator $\Phi : \scH \rightarrow \mathcal{Q}$ by $$ \Phi(f)= i \Lambda_{\omega} F_{A^f},$$ so that $A^f$ is a Hermitian Yang-Mills connection if and only if $\Phi(f)$ is a multiple of the identity.

\begin{lem}\label{linearisedoperator}\cite[Theorem 7.4.20]{SK} The differential $d\Phi_{\Id_E}: \mathcal{Q} \to \mathcal{Q}$ of $\Phi$ at the identity is given, for $a\in \mathcal{Q}$, by the Laplacian $$d\Phi_{\Id_E}(a) = \Delta_{A_{\End E}} a.$$
\end{lem}

Explicitly, we have
\begin{align}\label{laplacian}  \Delta_{A_{\End E}}\ = i \Lambda_{\omega} \left(  \partial_{A_{\End E}} \bar{\partial }_{A_{\End E}} - \bar{\partial}_{A_{\End E}} \partial_{A_{\End E}} \right).
\end{align}

Note that there is also another Laplacian operator, defined using the formal adjoint of $d_A$, which is given by \begin{align}\label{laplacian2}  \Upsilon_A \ = d_A^* d_A.
\end{align} The two are equal if and only if $A$ is weakly Hermitian Yang-Mills, i.e. satisfies equation \eqref{HEeqn} with the constant $c$  replaced by a function \cite[Lemma 1.2.5]{lubketeleman95}.

Recall that a holomorphic vector bundle is \emph{simple} if $E$ cannot be written as $A\oplus B$ for holomorphic vector bundles $A,B$ on $X$. We shall use this concept through the following:

\begin{lem}{\cite[Lemmata 1.2.5, 1.2.6]{lubketeleman95}}\label{laplaciansimple} Suppose the holomorphic structure $\bar{\partial}_A$ is simple. Then the kernel of the Laplacian $\Upsilon_{A}$ consists of constant multiples of the identity, and the image is the $L^2$-orthogonal complement to this kernel. In particular, $\Upsilon_A$ is an isomorphism when restricted to sections of $\End E$ with mean zero. \end{lem}

It is clear that, to prove Theorem \ref{blthm}, it is enough to prove it in the case the induced holomorphic structure on $E$ is simple. Thus for the remainder of the present work, we assume this is the case.

\section{Weighted spaces}\label{weighted-section}

Our aim in this section is to define weighted H\"older spaces for sections of vector bundles over $X \backslash \{p\}$ and $\Bl_pX$, and to study some of their properties. The spaces we define are completely analogous to the corresponding function spaces, as used in \cite{arezzopacard06,arezzopacardsinger11,szekelyhidi12}. 

We first recall the definition of  the H\"older space $C^{k,\alpha} (X, E)$. Since $X$ is compact, we can cover $X$ by a finite number of charts $U_i$. Picking a partition of unity $\varphi_i$ subordinate to this cover, for a section $s$ of $E$ we can consider $\varphi_i s$ as a map $U_i \subseteq \mathbb{C}^n \rightarrow \mathbb{C}^m$, where $m$ is the rank of $E$. For each $i$ one has the usual $C^{k,\alpha}$-norm $\|\varphi_is\|_{C^{k,\alpha}}$ of the map $\varphi_i s$, and we define $$\|s\|_{C^{k,\alpha}} = \max_i \|\varphi_is\|_{C^{k,\alpha}}.$$ This notion is independent of choice of atlas up to equivalence. 

Denote $X_p =X \backslash \{p\}$. For $C^{k,\alpha} (X_p, E)$ we make a similar definition, except we take one of the charts to be around $p$ and alter the definition of the $C^{k,\alpha}$-norm in this chart. We can assume this chart is a ball $B_2$ of radius two, and that the remaining charts cover the complement of the image of the ball $B_1$ of radius one. Over $B_2$, we can identify $E$ with $\mathbb{C}^m$, and so the restriction to $B_2^*=B_2\backslash \{0\}$ of a section $s$ of $E$ over $X_p$ can be considered as a map $B_2^* \rightarrow \mathbb{C} ^m$. For each $\delta \in \mathbb{R}$ and $r \in (0,1)$, we denote by $s_r^{\delta} : \bar{B}_2 \setminus B_1 \rightarrow \mathbb{C}^d$ the map 
\begin{align}\label{weightedsection} s_r^{\delta} (z) = r^{-\delta} s(rz).
\end{align}

\begin{definition}We define the $C^{k,\alpha}_{\delta}$\emph{-norm} on sections of $E$ over $X_p$ to be
\begin{align*} \| s \|_{C^{k,\alpha}_{\delta} (X_p,E)} =  \| s \|_{C^{k,\alpha} (X \setminus B_1,E)}  + \textnormal{sup}_{r\in(0,1)} \| s_r^{\delta} \|_{C^{k,\alpha} (B_2 \setminus B_1,\mathbb{C} ^m)},
\end{align*}
and let $C^{k,\alpha}_{\delta} (X_p, E)$ be the space of $C^{k,\alpha}_{\textnormal{loc}}$-sections of $E_{| X_p}$ that have finite $C^{k,\alpha}_{\delta}$-norm. \end{definition}

We next define weighted H\"older norms for sections of the trivial vector bundle of rank $m$ over the blowup $\pi:\Bl_0 \mathbb{C}^n\to \C^n$. Note that the region $\pi^{-1}(\mathbb{C}^n \setminus \{ 0 \}) $ in $\Bl_0 \mathbb{C}^n$ can be identified with $\mathbb{C}^n \setminus \{ 0 \} \subseteq \mathbb{C}^n$ via $\pi$. We will let $\zeta$ denote the coordinates on $\Bl_0 \mathbb{C}^n \setminus \pi^{-1} (0)$ arising from this identification. We will apply this in particular to the annular region $\pi^{-1} \left( B_2 \setminus B_1 \right) \cong B_2 \setminus B_1$. 

For a smooth map $f : \Bl_0 \mathbb{C}^n \rightarrow \mathbb{C} ^m$ we can define the weighted norm as 
\begin{align*} \| f \|_{C^{k,\alpha}_{\delta} (\Bl_0 \mathbb{C}^n, \mathbb{C} ^m)} = \| f \|_{C^{k,\alpha} (\pi^{-1}(B_1), \mathbb{C} ^m)} + \textnormal{sup}_{r \in(1, \infty)} \| f_r^{\delta} \|_{C^{k,\alpha} (B_2 \setminus B_1,\mathbb{C} ^m)} .
\end{align*} 
The space $C^{k,\alpha}_{\delta}(\Bl_0 \mathbb{C}^n, \mathbb{C} ^m)$ is then the space of $C^{k,\alpha}_{\textnormal{loc}}$-maps to $\mathbb{C} ^m$ which have finite such norm. 

We can also restrict this norm to subsets. In particular, for subsets of the form $\pi^{-1} (B_{a})$ with $a>2$, it becomes
\begin{align}\label{subsetnormbl0} \| f \|_{C^{k,\alpha}_{\delta} (\pi^{-1} (B_a), \mathbb{C} ^m)} = \| f \|_{C^{k,\alpha} (\pi^{-1}(B_1), \mathbb{C} ^m)} + \textnormal{sup}_{r \in(1, \frac{a}{2})} \| f_r^{\delta} \|_{C^{k,\alpha} (B_2 \setminus B_1,\mathbb{C} ^m)} .
\end{align} 
For each $a$ and $\delta$, this norm is equivalent to the usual $C^{k,\alpha}$-norm on $\pi^{-1}( B_a)$, but the constant of equivalency cannot be made independent of $a$.

We can now define the relevant H\"older norm for $\Bl_p X$ using a combination of the above two norms. Recall from the above that $z$ are the coordinates around $p$ and  $\zeta$ are the coordinates away from the exceptional divisor in $\Bl_0 \mathbb{C}^n$. We will identify the annular regions $ \varepsilon \leq |z|  \leq 1$ and $1 \leq |\zeta| \leq \varepsilon^{-1}$ by letting $\zeta = \varepsilon^{-1} z$, so we can think of this region as both a subset of $X$ and of $\Bl_0 \mathbb{C}^n$ via this coordinate change. Thus we think of $\Bl_p X$ as consisting of three regions:
\begin{itemize} \item $\Bl_p X \setminus \pi^{-1} \left( B_1 (p) \right)$, which can be identified with $X \setminus \left( B_1 (p) \right)$.
\item $\pi^{-1} \left( B_1(p) \setminus B_{\varepsilon} (p) \right)$, which is the annular region described above. 
\item $\pi^{-1} \left( B_{\varepsilon} (p) \right)$, which is the region $|\zeta| \leq 1$ in $\Bl_0 \mathbb{C}^n$.
\end{itemize}

For a section $s$ of $E$ over $\Bl_p X$, we will let $\tilde{s}_{\varepsilon}^{\delta}$ denote the map 
\begin{align*}\tilde{s}_{\varepsilon}^{\delta}: \zeta \mapsto \varepsilon^{- \delta} s ( \varepsilon \zeta)
\end{align*}
from the region $| \zeta | \leq 1$ in $\Bl_0 \mathbb{C}^n$ to $\mathbb{C} ^m$. We also let $s_r^{\delta}$ be the section of $E$ over $B_2 (p) \setminus B_1(p) \subseteq X$ given as in equation (\ref{weightedsection}).

 \begin{definition} We define the $C^{k,\alpha}_{\delta}$\emph{-norm} on sections of $E$ over $\Bl_p X$ to be
\begin{align*} \| s \|_{C^{k,\alpha}_{\delta} (\Bl_p X , E)} = \| & s \|_{C^{k,\alpha}(X \setminus B_1, E)} + \\ &+ \sup_{r\in (\varepsilon, 1)} \| s_r^{\delta} \|_{C^{k,\alpha} (B_2 \setminus B_1, E)} + \| \tilde{s}^{\delta}_{\varepsilon} \|_{C^{k,\alpha}_{\delta} (\pi^{-1}(B_{\varepsilon} (p)), \mathbb{C} ^m)}.
\end{align*}\end{definition}
An important point to note is that the norm depends on $\varepsilon$. An equivalent description, as in the function case \cite[p167]{szekelyhidi14book}, is as follows. Let $r_{\varepsilon} = \varepsilon^{\kappa}$ where $\kappa = \frac{n-1}{n}$. Let $\gamma$ be a smooth function $\mathbb{R} \rightarrow [0,1]$ which is zero for $x \leq 1$ and $1$ for $x \geq 2$. Define $\gamma_1 (z) = \gamma (\frac{|z|}{r_{\varepsilon}})$ and $\gamma_2 = 1- \gamma_1$. Then an equivalent norm is given by: \begin{align*} \| s \|_{C^{k,\alpha}_{\delta}(\Bl_p X,E)} = \| \gamma_1 s \|_{C^{k,\alpha}_{\delta} (X_p, E)} + \varepsilon^{-\delta} \| \gamma_2 s \|_{C^{k,\alpha}_{\delta} (\pi^{-1}(B_1), \mathbb{C} ^m)}.
\end{align*}

As in the function case \cite[p. 167]{szekelyhidi14book}, the $C^{k,\alpha}_{\delta}$ spaces do not actually depend on $\delta$, as they simply consist of locally $C^{k,\alpha}$ sections. Instead the weight $\delta$ only affects the norm. Again as in the function case, for two weights $\delta < \delta'$ one has the inequalities
\begin{align}\label{weightcomparison} \| s\|_{C^{k,\alpha}_{\delta} (\Bl_p X , E)}  \leq  \| s \|_{C^{k,\alpha}_{\delta'} (\Bl_p X , E)} \leq \varepsilon^{\delta - \delta'} \| s\|_{C^{k,\alpha}_{\delta} (\Bl_p X , E)}.
\end{align}

We will also briefly require weighted Sobolev spaces, although only on $X \setminus \{p\}$. Let $\rho$ be a smooth positive function on $X_p$, which equals the radius function $r$ in $B_1^*(p)$ and $2$ on the complement of $B_2(p)$. 

\begin{definition} We define the

\begin{enumerate}[(i)]
\item  $L^2_{\delta}$\emph{-norm} by $$\| s\|^2_{L^2_{\delta}(X_p, E)} = \int_{X_p} |s|_h^2 \rho^{-\delta} \omega^n,$$
\item and the $W^{2,k}_{\delta}$\emph{-norm} by
$$\| s\|_{W^{2,k}_{\delta}(X_p, E)} = \sum_{j=0}^k \| \nabla^k s \|_{L^2_{\delta-j} (X_p)},$$\end{enumerate}
where we are using the Hermitian metric $h$ and the K\"ahler metric $\omega$ to calculate the pointwise norm $|\nabla^k s|$ of $\nabla^k s$. We then define $W^{2,k}_{\delta} (X_p, E)$ to be the completion of smooth compactly supported sections of $E$ on $X_p$ under this norm. \end{definition}

We now describe main properties of the various weighted norms that we will require. 

\begin{lem}\label{productembedding} Let $E,F$ be bundles over $X$. For the weighted spaces on $Y = X_p$, $Y =\Bl_0 \mathbb{C}^n$ or $Y = \Bl_p X$, the map 
\begin{align*} C^{k,\alpha}_{\delta} (Y, E) \times C^{k,\alpha}_{\delta'} (Y , F) \rightarrow C^{k,\alpha}_{\delta + \delta'} (Y, E \otimes F)
\end{align*}
given by taking the tensor product of sections is continuous. In the case of $Y = \Bl_p X$, the constant such that 
\begin{align*} \| s \otimes t \|_{ C^{k,\alpha}_{\delta + \delta'} (Y, E \otimes F)}   \leq c \| s \|_{ C^{k,\alpha}_{\delta} (Y, E)  } \cdot \| t \|_{C^{k,\alpha}_{\delta'} (Y , F) }
\end{align*}
can be chosen independently of $\varepsilon$. 
\end{lem}
\begin{proof} The $C^0$-estimates on the scaled regions hold as the $(i,j)^{\textnormal{th}}$-component $(s \otimes t)_{(i,j)}$ of $s \otimes t$ satisfies
\begin{align*}  (r^{- (\delta + \delta')} s \otimes t)_{(i,j)} (rz) &= r^{- \delta} s(rz) \cdot r^{-\delta'} t(rz) \\
& \leq \| s \|_{C^{k,\alpha}_{\delta} (Y)} \cdot \| t\|_{C^{k,\alpha}_{\delta'} (Y)}.
\end{align*} For the higher estimates, applying the Leibniz rule gives \begin{align*} \left( r^{-(\delta + \delta') } \nabla ( s \otimes t) \right)_{(i,j)} (rz) &= ( r^{- \delta} \nabla s (rz) ) r^{-\delta'} t (rz) ) +( r^{- \delta}  s (rz) ) r^{-\delta'} \nabla t (rz) ) \\
& \leq 2 \| s \|_{C^{k,\alpha}_{\delta} (Y)} \cdot \| t\|_{C^{k,\alpha}_{\delta'} (Y)},
\end{align*} which implies the required estimates. The argument for the estimate in the region $\pi^{-1} (B_{\varepsilon})$ when $Y= \Bl_p X$ follows in exactly the same manner.

The last part concerning the case  of $Y= \Bl_p X$ follows because the constants above do not depend on what range $r$ lies in, and so are independent of $\varepsilon$. 
\end{proof}

We also need to know when the H\"older spaces are guaranteed to be contained in the Sobolev spaces.
\begin{lem}\label{holderinsobolev} We have an inclusion
\begin{align*} C^{k,\alpha}_{\delta} (X_p, E) \subseteq W^{2,k}_{\delta'} (X_p, E)
\end{align*}
if and only if $\delta' < 2 \delta + 2n$.
\end{lem}

\begin{proof}This follows immediately from the expression of the volume form in $\mathbb{R}^{2n}$ in spherical coordinates. \end{proof}

\section{Linear theory}\label{linear-section}

As we have seen in Lemma \ref{laplaciansimple}, when $E$ is simple, the operator $\Upsilon_A$ on $\End E$ has kernel and cokernel given by constant multiples of the identity. In particular, once it is shown that $\pi^* E$ is simple (see Proposition \ref{piEsimple}), this holds for the pullback bundle over the blowup. Taking care of the cokernel, we can then construct a one-sided inverse of $\Upsilon_{A_{\epsilon}}$ associated to the metric $\omega_{\varepsilon}$, for each $\varepsilon$. The goal of this section is to prove Theorem \ref{blowuplinearthm}, which gives a corresponding statement along with bounds for these inverses of the Laplacian $\Delta$ independently of $\varepsilon$ in the weighted norms. We prove this by applying similar results for the Laplacian on $X_p$ and $\Bl_0 \mathbb{C}^n$, so we begin by considering these two cases. Note that on the blow-up, $\Upsilon_{A_{\epsilon}}$ and $\Delta$ now differ, but  their difference becomes small when $\varepsilon \to 0$.

\subsection{Linear theory on $X_p$}\label{puncturedlinear}

The goal of this section is to prove a Fredholm Theorem for the Laplace operator $\Delta$ on $\End E \rightarrow X_p$ associated to the K\"ahler metric $\omega$ on $X_p$ and the Chern connection of the Hermitian metric $h$.  To emphasise the weighted spaces we are working in, we will write the Laplacian $$\Delta_{\delta} : C^{k+2, \alpha}_{\delta} (X_p,\End E) \rightarrow C^{k,\alpha}_{\delta - 2} (X_p,\End E)$$ with a subscript to denote the weight. Our aim is to prove:

\begin{thm}\label{mainlinearthm} The operator $\Delta_{\delta}$ is Fredholm apart from the discrete sets of indicial roots  $\mathbb{Z} \setminus (2-2n,0)$. Moreover, if $\delta$ is not an indicial root, then
\begin{align*} \Image ( \Delta_{\delta} ) = \textnormal{Ker } ( \Delta_{ 2 - 2 n - \delta})^{\perp}.
\end{align*}
\end{thm}

The strategy is to show that near $p$ the operator is a perturbation of the Euclidean Laplace operator, which simply acts component-wise on endomorphisms of $E$. We then apply the Lockhart-McOwen theory and regularity theory to finish the proof. 

 Pick a cutoff function $\gamma$ which equals one on $B_{1}$ and zero on $X\backslash B_2$. Then for a section $s\in C^{k,\alpha}_{\delta}(X_p,\End E)$, by extending by zero we can consider $\gamma s$ to be a section $\gamma s \in C^{k,\alpha}_{\delta}(B_2^*, \mathbb{C}^{m^2}).$ For $\gamma s$, we can compute  its Laplacian also with respect to the Euclidean Laplacian $\Delta_{\Euc}$ induced by the flat metric on the trivial bundle over $B_2^*$.

\begin{lem}\label{laplacian-comparison}  We have $$\Delta_{\delta}s - \Delta_{\Euc} \gamma s \in C^{k-2,\alpha}_{\delta}.$$ \end{lem}

\begin{proof}
We begin by finding the difference between the two Laplace operators. We can identify the total space of $E_{|B_2(p)}$ with $B_2(0) \times \mathbb{C} ^m \subseteq \mathbb{C}^n \times \mathbb{C} ^m$. We will let $z_1, \cdots, z_n$ be the coordinates on the base, and $w_1, \cdots, w_m$ be the coordinates on the fibres. A section of $\End E$ over $B_2$ can be identified with a matrix $(\psi_{ij})$, say, where $i,j  \in \{ 1, \cdots, m \}$. If $s$ is a section of $E$, we can write it as $s(z) = (s_1(z), \cdots, s_r(z)) \in \mathbb{C} ^m$ over $B_2$. 

The covariant derivative $d_A$ associated to a connection $A$ on $E$ can be written on $B_2$ as 
\begin{align*} d_A = d + \mathcal{A}
\end{align*}
for some $\mathcal{A} \in \Lambda^1( \mathbb{C} ^m)$. In our case, this is the Chern connection of $h$, and so $\mathcal{A}$ is in fact a $(1,0)$-form.  The induced covariant derivative $d_{A_{\End E}}$ on $\End E$ is given by 
\begin{align*}d_{A_{\End E}} ( \psi) &= d_A \circ \psi - \psi \circ d_A,
\end{align*}
which near $p$ equals
\begin{align*} d \circ \psi - \psi \circ d + \mathcal{A} \circ \psi - \psi \circ \mathcal{A} .
\end{align*}
By the Leibniz rule we have  
\begin{align*} d\left(\sum_{j} \psi_{ij} s_j \right) &= \sum_{j} \left( d(\psi_{ij}) s_j + \psi_{ij} ds_j \right) \\
& \in \mathbb{C} ^m \otimes \Lambda^1 (\mathbb{C}^n).
\end{align*}
Thus near $p$, the covariant derivative $d_{A_{\End E}} \left( (\psi_{ij}) \right) (s) $ for a section $s$ of $E$ has $l^{\textnormal{th}}$ component equal to 
\begin{align*} \left( d_{A_{\End E}} \left( (\psi_{ij}) \right) (s) \right)_l &= \sum_{j} \left( d(\psi_{lj}) s_j  + [\mathcal{A}, (\psi_{pq})]_{lj} s_j\right).
\end{align*}
so
\begin{align*} d_{A^{\End}} \left( (\psi_{ij}) \right) &=  \left( d(\psi_{ij})  + [\mathcal{A}, (\psi_{pq})]_{ij} \right) .
\end{align*}

Next note that 
\begin{align*} d_A &= \partial_A + \bar{\partial}_A \\
&= (\partial + \mathcal{A}) + \bar{\partial}
\end{align*}
is the decomposition of $d_A$ into its $(1,0)$ and $(0,1)$ components, since $A$ is the Chern connection of $E$. Thus
\begin{align*} \partial_A \bar{\partial}_A (\psi) &= \partial_A \left( \bar{\partial} (\psi_{ij})  \right) \nonumber \\
&= \partial \bar{\partial} (\psi_{ij})+ [\mathcal{A}, \bar{\partial} (\psi_{pq})]_{ij}  .
\end{align*}
and
\begin{align*} \bar{\partial}_A \partial_A (\psi) &= \bar{\partial}_A \left( \partial (\psi_{ij})  + [\mathcal{A}, (\psi_{pq})]_{ij} \right) \nonumber \\
&= - \partial \bar{\partial} (\psi_{ij}) + \bar{\partial} \left( [\mathcal{A}, (\psi_{pq})]_{ij} \right)  .
\end{align*}
Since $$[\mathcal{A}, \bar{\partial} (\psi )] - \bar{\partial} \left( [\mathcal{A}, \psi ] \right)  = [\psi , \bar{\partial} \mathcal{A} ],$$ we thus have from equation (\ref{laplacian}) that 
\begin{align}\label{pertop} \Delta_{A_{\End E} } ( \psi) = \Lambda_{\omega} \left(  2  \partial \bar{\partial} (\psi) + [\psi , \bar{\partial} \mathcal{A} ] \right) .
\end{align}

The first term is simply the Laplacian associated to the metric $\omega$ acting componen-twise. Since we are working in holomorphic normal coordinates with resepect to $\omega$, it follows from the claim in the function case (see e.g. \cite[p. 162]{szekelyhidi14book}) that the difference between this operator and the Euclidean Laplace operator acting component-wise lies in $C^{0,\alpha}_{\delta}$. We thus need to show that the operator $$\psi \mapsto [ \psi, \bar{\partial} \mathcal{A}]$$ maps $C^{2,\alpha}_{\delta}$ to $C^{0,\alpha}_{\delta}$. For this we use the standard local formula for $\mathcal{A}$ as
\begin{align*} \mathcal{A} = \bar{h}^{-1} \partial\bar{h},
\end{align*}
see for example \cite[p. 177]{huybrechts05} Thus
\begin{align*} \bar{\partial} \mathcal{A} &= \bar{h}^{-1} \bar{\partial} \partial\bar{h}+  \bar{\partial} \left( \bar{h}^{-1} \right) \wedge \partial\bar{h} . 
\end{align*}
Since $h$ lies in $C^{k,\alpha}_{0}$ for any $k$, it follows from the multiplicative properties of the weighted norms that $\Lambda_{\omega} [\psi, \bar{\partial} \mathcal{A}] \in C^{k,\alpha}_{\delta}$ whenever $\psi$ is.
\end{proof} 

Remark that a priori $\Delta_{\delta}s - \Delta_{\Euc} s \in C^{k-2,\alpha}_{\delta-2}$. The improvement given in the Lemma above will enable us to reduce to the Euclidean setting in our proof of Theorem \ref{mainlinearthm}. \color{black}

\begin{proof}[Proof of Theorem \ref{mainlinearthm}] By using Lemma \ref{laplacian-comparison}, we can apply the theory of Lockhart-McOwen \cite{lockhartmcowen85}. The Fredholm property in Theorem \ref{mainlinearthm} is then a consequence of this: $\Delta_{\delta}$ is Fredholm if and only if the Euclidean Laplace operator is, which is true if and only if $\delta$ is not an indicial root of the Laplacian.

For the characterisation of the image in terms of the kernel with a different weight, we apply the regularity theory in weighted spaces. This says that if we solve $\Delta_{\delta} (s) = t$ in the sense of distributions, with $t \in C^{0,\alpha}_{\delta - 2}$ and $s \in L^2_{\delta'}$ such that $\delta' < 2 \delta + 2n$, then $s \in C^{2,\alpha}_{\delta}$. 

We can identify the dual of $L^2_{\delta'}$ with $L^2_{-\delta'}$ via the $L^2$-inner product. The image of $\Delta_{\delta'}$ in $L^2_{\delta'}$ is then the orthogonal complement of the kernel of the adjoint 
\begin{align*} \Delta_{\delta'}^*: L^2_{-\delta'} = \left( L_{\delta'}^{2} \right)^* \rightarrow \left( W^{2,2}_{\delta' + 2}\right)^*
\end{align*}
of $\Delta_{\delta'}$. But if $s \in L^2_{-\delta'}$ is in this kernel, then it solves
\begin{align*} \Delta_{\delta'} (s) = 0
\end{align*}
in the sense of distributions. So by the regularity theory, $s$ will lie in any weighted space $C^{2,\alpha}_{\tau}$ such that $C^{2,\alpha}_{\tau} \subset L^2_{- \delta'}$, i.e any $\tau$ satisfying $2 \tau + 2n > - \delta'$.

We apply this to $\delta' = 2 \delta + 2n - 4 - \varepsilon$. Then $C^{2,\alpha}_{\delta} \subset W^{2,2}_{\delta' + 2}$ and moreover
\begin{align*} \Image \left( \Delta_{\delta} :  C^{2,\alpha}_{\delta} \rightarrow C^{0,\alpha}_{\delta - 2} \right) = \Image \left(  \Delta_{\delta} :  W^{2,2}_{\delta' + 2} \rightarrow L^2_{\delta' } \right) \cap C^{0,\alpha}_{\delta - 2},
\end{align*}
again by the regularity theory. So by the discussion of the previous paragraph, picking $\tau = \frac{1}{2}( \varepsilon - 2n - \delta' ) =  \varepsilon + 2 - 2n - \delta$, we have 
\begin{align}\label{almostresult} \Image \left( \Delta_{\delta} :  C^{2,\alpha}_{\delta} \rightarrow C^{0,\alpha}_{\delta - 2} \right) =   \left( \ker \Delta_{\varepsilon + 2 - 2n - \delta} \right)^{\perp} \cap C^{0,\alpha}_{\delta - 2}.
\end{align}
But since $\delta$ is not an indicial root, neither is $2-2n- \delta$, and so there is an open interval about $2-2n - \delta$ for which the kernel of $\Delta_{\delta}$ does not change. In particular, when $\varepsilon$ is chosen sufficiently small, we may replace $\varepsilon + 2 - 2n - \delta$ with $2 - 2n - \delta$ in equation (\ref{almostresult}), which is what we wanted to show.
\end{proof}

The weights relevant to us are $\delta \in (2-2n, 0)$. Applying Theorem \ref{mainlinearthm} to weights in this range, we see that the kernel of $\Delta_{\delta}$ equals the kernel with weight $0$, which is simply the kernel of $\Delta_{\delta}$ on the whole of $X$, which is just $\mathbb{C} \cdot \Id_E$. Since $2 -2n - \delta \in (2-2n, 0)$ if $\delta \in (2-2n, 0)$, it follows that the image is the orthogonal complement to the same kernel, i.e. the elements of $C^{0,\alpha}_{\delta-2}$ whose average is $0$. We have just proved:
\begin{cor}\label{relevantweights} Let $\delta \in (2-2n,0)$. Then $\Delta_{\delta}$ is Fredholm with kernel $\mathbb{C} \cdot \Id_E$ and image
\begin{align*} \Image \Delta_{\delta} = \left( \mathbb{C} \cdot \Id_E \right)^{\perp} \subset C^{0,\alpha}_{\delta-2} (X_p , \End E).
\end{align*}
\end{cor}

\subsection{Linear theory on $\Bl_0 \mathbb{C}^n$}\label{localblowuplinear}
On $\Bl_0 \mathbb{C}^n$ we require the following result regarding the Laplacian operator acting on sections of the trivial vector bundle.
\begin{thm} Suppose $\delta \notin \mathbb{Z} \setminus (2-2n)$. Then the Laplacian operator $$\Delta_{\delta} : C^{k+2,\alpha}_{\delta} (\Bl_0 \mathbb{C}^n , \mathbb{C} ^m ) \rightarrow C^{k,\alpha}_{\delta -2} (\Bl_0 \mathbb{C}^n , \mathbb{C} ^m )$$ of the rank $r$ trivial vector bundle with its flat metric over $\Bl_0 \mathbb{C}^n$  with the Burns-Simanca metric is Fredholm. Moreover,
\begin{align*} \Image \Delta_{\delta} = \left( \textnormal{Ker } \Delta_{2-2n-\delta} \right)^{\perp} .
\end{align*}
\end{thm}
The proof is to reduce to the function case. This is possible since we are using the flat Hermitian metric on the trivial bundle. The Chern connection then simply becomes the operator $d$ acting component-wise, and so the Laplacian operator associated to the flat metric is simply the Laplacian operator on functions associated to the base metric acting on each component.  For the result in the function case, see \cite[Theorem 8.6]{szekelyhidi14book} and the references therein.

The weights which will be relevant to us are the negative weights. For these we require the following.
\begin{cor}\label{relevantwtsblowupcn} Suppose $\delta <0$. Then the kernel of  $$\Delta_{\delta} : C^{k+2,\alpha}_{\delta} (\Bl_0 \mathbb{C}^n , \mathbb{C} ^m ) \rightarrow C^{k,\alpha}_{\delta - 2} (\Bl_0 \mathbb{C}^n , \mathbb{C} ^m )$$ is trivial.
\end{cor}
Again this is a direct consequence of the result for the case $m=1$.

\subsection{Linear theory on $\Bl_p X$}\label{blowuplinear}

On $\Bl_p X$ we are considering the bundle $\pi^* E$, where $\pi : \Bl_p X \rightarrow X$ is the blow-down map. To simplify notation, we will also denote the pullback bundle by $E$. We will let $\omega_{\varepsilon}$ and $h_{\varepsilon}$ be a base metric and bundle metric, respectively, which will both depend on $\varepsilon$ and which we will now define. 

The base metric is defined as in \cite{szekelyhidi12}. This uses the Burns-Simanca metric on $\Bl_0 \mathbb{C}^n$ whose associated K\"ahler form $\eta$ can be written as 
\begin{align*} \eta = i \partial \bar{\partial} ( |\zeta|^2 + \psi ( \zeta ) ),
\end{align*}
where $\psi = O(|\zeta|^{4-2n})$ if $n>2$ and $\psi = \log(|\zeta|)$ if $n=2$, where $\zeta = \varepsilon^{-1} z$. Recall also that $\omega$ can be written as 
\begin{align*} \omega = i \partial \bar{\partial} ( |z|^2 + \varphi(z) ),
\end{align*}
where $\varphi = O (|z|^4)$.

Recall that $r_{\varepsilon} = \varepsilon^{\kappa}$ where $\kappa = \frac{n-1}{n}$. We have let $\gamma$ be a smooth function $\mathbb{R} \rightarrow [0,1]$ which is zero for $x \leq 1$ and $1$ for $x \geq 2$. Further we let $\gamma_1 (z) = \gamma (\frac{|z|}{r_{\varepsilon}})$ and let $\gamma_2 = 1- \gamma_1$. The metric $\omega_{\varepsilon}\in [\pi^*\omega] - \epsilon c_1(F)$ on $\Bl_p X$ is defined to be 
\begin{align}\label{basemetric} \omega_{\varepsilon}=
  \begin{cases}
   \omega      & \quad \text{on } X \setminus B_{2 r_{\varepsilon}} \\
    i \partial \bar{\partial} \left( |z|^2 + \gamma_1 (z) \varphi (z) + \varepsilon^2 \gamma_2 (z) \psi (\varepsilon^{-1} z) \right)     & \quad \text{on }  B_{2 r_{\varepsilon}} \setminus B_{r_{\varepsilon}} \\
    \varepsilon^2 \eta      & \quad \text{on } \pi^{-1} (B_{r_{\varepsilon}} ) \\
  \end{cases}
\end{align}
The point of choosing $\kappa$ to be $\frac{n-1}{n}$ is that the K\"ahler potential above then is $O(|z|^4)$, independently of $\varepsilon$, in the annular region. In other words, we have that
\begin{align}\label{potentialbound}  \gamma_1 (z) \varphi (z) + \varepsilon^2 \gamma_2 (z) \psi (\varepsilon^{-1} z) = O(|z|^4),
\end{align}
where the $O(|z|^4)$-term is understood to be independent of $\varepsilon$. 

The bundle metric $h_{\varepsilon}$ is defined similarly, using the flat metric $h_{\textnormal{flat}}$ on the trivial rank $r$ bundle over $\Bl_0 \mathbb{C}^n$. More precisely, if $h$ denotes the bundle metric on $E \rightarrow X$ and $\gamma_1$ and $\gamma_2$ are the cut-off functions of the previous section, then we set
\begin{align}\label{bundlemetric} h_{\varepsilon}=
  \begin{cases}
    h      & \quad \text{on } X \setminus B_{2 r_{\varepsilon}} \\
    \gamma_1 h + \gamma_2 h_{\textnormal{flat}}     & \quad \text{on }  B_{2 r_{\varepsilon}} \setminus B_{r_{\varepsilon}} \\
    h_{\textnormal{flat}}      & \quad \text{on } \pi^{-1} (B_{r_{\varepsilon}} ) \\
  \end{cases}
\end{align}

We will later use the following bound on the $\gamma_i$:
\begin{align}\label{cutoffbound} \| \gamma_i \|_{C^{4,\alpha}_0 (\Bl_p X)} \leq c.
\end{align}

The first operator of interest in this section is essentially the Laplace operator $\Delta_{\varepsilon}=\Delta_{\epsilon,\End E}$ associated to the Chern connection $A_{\varepsilon,\End E}$ for the Hermitian metric $h_{\varepsilon}$ on $\pi^*E \rightarrow \Bl_p X$. Letting $q$ be some arbitrary, but fixed, point in $X \setminus \{ p \}$,  the main result is the following. 
\begin{thm}\label{blowuplinearthm} Let $\delta \in (2-2n, 0)$. Then the operator $\tilde{\Delta}_{\varepsilon} : C^{2,\alpha}_{\delta} (\Bl_p X, \End E ) \rightarrow C^{0,\alpha}_{\delta-2} (\Bl_p X, \End E )$ given by
\begin{align*} f \mapsto \Delta_{\varepsilon} (f) - \tr_q (f) \cdot \Id_E
\end{align*}
is an isomorphism. Moreover, the inverse of $\tilde{\Delta}_{\varepsilon}$ has bounded operator norm, $\| \tilde{\Delta}_{\varepsilon}^{-1}  \| \leq C$, independently of $\varepsilon$. 
\end{thm}

The proof of this uses an analogous technique to the work of Biquard-Rollin \cite{BR}. We follow the strategy of Sz\'ekelyhidi in the context of the linearisation of the scalar curvature on K\"ahler manifolds \cite[Theorem 8.14]{szekelyhidi14book}.

Note that the main content of this theorem is the bound on the inverse of $\tilde{\Delta}_{\varepsilon}$. It will follow from the simplicity of $\pi^* E$ that $\tilde{\Delta}_{\varepsilon}$ is an isomorphism as $\Bl_p X$ is compact and the spaces $C^{k,\alpha}_{\delta}$ are nothing but the usual $C^{k,\alpha}$-spaces with a different, but equivalent norm, for each $\varepsilon$. To ensure that the vector bundle Laplacian is invertible for $\pi^*E$, we use the following. 

\begin{prop}\label{piEsimple} Suppose $E$ is a simple vector bundle over a complex manifold $X$, and let $p: Y\to X$ be a bimeromorphic morphism. Then $p^*E$ is simple. \end{prop}

\begin{proof} The largest analytic $Z\subset X$ such that $Y\backslash p^{-1}(Z)\cong X\backslash Z$ satisfies $\dim Z\leq \dim X - 2$ since $X$ is smooth and $p$ is bimeromorphic. Suppose $p^*E \cong A\oplus B$ for holomorphic subbundles $A,B$. The direct images $p_*A, p_*B$ are coherent by Grauert's Theorem, and on $X\backslash Z$ we have an isomorphism $p_*A\oplus p_*B \cong E$. Then since $p_*A\oplus p_*B $ and $E$ are coherent sheaves which are isomorphic away from $Z$, which has codimension two, they must be isomorphic on all of $X$ as $X$ is smooth. This contradicts the fact that $E$ was assumed to be simple, giving the result.
 \end{proof}
 
We apply this to the map $\pi: \Bl_p X \to X$ to conclude that $\pi^*E$ is a simple holomorphic vector bundle. It then easily follows as above that, after adding the term $\tr_q \cdot \Id_E$, the similarly defined operator $\tilde{\Upsilon}_{\varepsilon}$ is an isomorphism. To conclude that $\tilde{\Delta}_{\varepsilon}$ also is an isomorphism, we would then need to show that the difference between the two operators becomes small as $\varepsilon \to 0$. We will however just consider $\tilde{\Delta}_{\varepsilon}$ directly, but it will be clear (see the proof of Lemma \ref{approxsol}) that this difference does indeed go to $0$ with $\varepsilon$.

Before beginning the proof of Theorem \ref{blowuplinearthm}, we will first prove the following Schauder estimate. For ease of notation, we write $C^{k,\alpha}_{\delta}$ for $C^{k,\alpha}_{\delta} (\Bl_p X, \End E)$ throughout the proof.
\begin{lem} There exists a constant $c>0$, independently of $\varepsilon$, such that
\begin{align}\label{firstschauder} \| s \|_{C^{2,\alpha}_{\delta}} \leq c \left( \| \tilde{\Delta}_{\varepsilon} (s)\|_{C^{0,\alpha}_{\delta-2}} + \| s\|_{C^{0}_{\delta}} \right).
\end{align}
\end{lem}
\begin{proof} First note that on the region away from the exceptional divisor, the usual Schauder estimates imply that we have the bound $$ \| s \|_{C^{2,\alpha} (X \setminus B_1)} \leq c \left( \| \tilde{\Delta}_{\varepsilon} (s)\|_{C^{0,\alpha} (X \setminus B_1) } + \| s\|_{C^{0} (X \setminus B_{\frac{1}{2}}) } \right) ,$$ and the restriction of the weighted norm to $M \setminus B_{\frac{1}{2}}$ is equivalent to the usual H\"older norm with the constant of equivalency depending only on $\delta$, not $\varepsilon$. Note that the operator $\tilde{\Delta}_{\varepsilon}$ is the same for all values of $\varepsilon < \frac{1}{2}$, so $c$ is independent of $\varepsilon$.

Secondly, the preimage of $B_{\varepsilon}$ is the region $|\zeta| \leq 1$ thought of as a subset of $\Bl_0 \mathbb{C}^n$. Since the annular region on which we are matching the K\"ahler metrics and Hermitian metrics gets larger and larger thought of as a subset of $\Bl_0 \mathbb{C}^n$, for $\varepsilon >0$ sufficiently small the metrics on the region $|\zeta| \leq 2$ are always the Burn-Simanca metric and the flat Hermitian metric. Thus we get a bound$$ \| s \|_{C^{2,\alpha} ( \{ |\zeta| < 1 \} )} \leq c \left( \| \tilde{\Delta}_{\varepsilon} (s)\|_{C^{0,\alpha} ( \{ |\zeta| < 1 \} ) } + \| s\|_{C^{0} ( \{ |\zeta | < 2 \} ) } \right) ,$$ with $c$ independent of $\varepsilon$.  When we scale, we introduce a factor of $\varepsilon^2$ on $\tilde{\Delta}_{\varepsilon} (s)$, which implies that the above inequality gives us that we have a bound $$ \| s \|_{C^{2,\alpha}_{\delta} ( \pi^{-1} (B_{\varepsilon} ) )} \leq c \left( \| \tilde{\Delta}_{\varepsilon} (s)\|_{C^{0,\alpha}_{\delta - 2} (\pi^{-1} (B_{\varepsilon} ) ) } + \| s\|_{C^{0}_{\delta} ( \pi^{-1} (B_{2 \varepsilon} ) ) } \right) ,$$ 
with $c$ independent of $\varepsilon$. 

Finally, for the annular region, note that we could define an equivalent norm to the $C^{k,\alpha}_{\delta}$-norm by 
\begin{align}\label{equivalentnorm} \|  s \|_{C^{k,\alpha}(X \setminus B_1, E)} + \sup_{r\in (\varepsilon, 1)} \| s_r^{\delta} \|_{C^{k,\alpha} (B_3 \setminus B_{\frac{1}{2}}, E)} + \| \tilde{s}^{\delta}_{\varepsilon} \|_{C^{k,\alpha}_{\delta} (\pi^{-1}(B_{\varepsilon} (p)), \mathbb{C} ^m)},
\end{align}
i.e. we are enlarging the annular regions in the middle term. The constant of equivalency depends on $\delta$, but \textit{not} on $\varepsilon$. Next recall that $\omega_{\varepsilon}$ and $h_{\varepsilon}$ are uniformly equivalent to the Euclidean metric and the flat metric, respectively, on the annular region. Thus $r^{-2} \tilde{\Delta}_{\varepsilon} (s^{\delta}_r )$ has a uniform estimate on the $C^{k,\alpha}$-norm of its components and its constant of ellipticity. The factor $r^{-2}$ arises because $\tilde{\Delta}_{\varepsilon}$ is a second order operator.

 In particular, the local Schauder estimates imply that
\begin{align*}\| s_r^{\delta} \|_{C^{2,\alpha} (B_2 \setminus B_{1}, E)}  \leq C \left( \| r^{-2} \tilde{\Delta}_{\varepsilon} (s_r^{\delta}) \|_{C^{0,\alpha} (B_2 \setminus B_{1}, E)}  +  \| s_r^{\delta} \|_{C^{0} (B_3 \setminus B_{\frac{1}{2}}, E)} \right), 
\end{align*}
where $C$ is independent of $\varepsilon$. Taking supremums and using that equation (\ref{equivalentnorm}) defines an equivalent norm with constant of equivalency independent of $\varepsilon$, we precisely obtain the estimate (\ref{firstschauder}).
\end{proof}

We will now improve the above bound to show that the term $\| s\|_{C^{0}_{\delta}}$ is superfluous. This precisely gives a bound on the operator norm of $\tilde{\Delta}_{\varepsilon}^{-1}$, which is what is needed to complete the proof of Theorem \ref{blowuplinearthm}. 
\begin{proof}
We will prove that we do not need the term $\| s\|_{C^{0}_{\delta}}$ in the estimate (\ref{firstschauder}) by contradiction. So suppose this cannot be removed. Then we can find sequence $\varepsilon_i >0$ tending to zero and sections $s_i$ of $\End E$ over $\Bl_p X$  with $\| s_i \|_{C^{2,\alpha}_{\delta}} = 1$ and 
\begin{align}\label{Lbound}  \| \tilde{\Delta}_{\varepsilon_i} (s_i)\|_{C^{0,\alpha}_{\delta-2}} < \frac{1}{i}.
\end{align}
This implies that there is a $C>0$ such that
\begin{align*} \left| \int_{\Bl_p X} \langle \tilde{\Delta}_{\varepsilon_i} (s_i ), \Id_E \rangle \omega_{\varepsilon_i}^n \right| < \frac{C}{i}.
\end{align*}
Since $\Delta$ is self-adjoint and contains $\mathbb{C}\cdot \Id_E$ in its kernel, the integral of $\langle \Delta (s_i ), \Id_E \rangle $ is $0$, and so we obtain the estimate
\begin{align*} |\tr_q (s_i) | < \frac{C}{i}.
\end{align*}

Since our sequence satisfies $\| s_i \|_{C^{2,\alpha}_{\delta}} = 1$ on $\Bl_p X$, we have uniform bounds for the $s_i$ on the regions $X \setminus B_{\varepsilon_i}$. By Arzela-Ascoli, for any $\beta < \alpha$ we can extract a subsequence that converges locally in $C^{4,\beta}_{\delta}$ to a section $s$ of $\End E$ over $X_p$. By the above, this section satisfies that $\tr_q (s) = 0$ and that $\tilde{\Delta}_{\omega} (s) =0$. It then follows from Corollary \ref{relevantweights} that $s$ is $0$.

Next note that equations (\ref{firstschauder}) and (\ref{Lbound}) imply 
\begin{align*} 1 \leq \frac{c}{i} + c \| s_i \|_{C^{0}_{\delta}}.
\end{align*}
Thus 
\begin{align*}\frac{1}{c} - \frac{1}{i} \leq \| s_i \|_{C^{0}_{\delta}},
\end{align*}
which in turn implies that the sequence $s_i$ has $C^{0}_{\delta}$-norm bounded below away from zero. Moreover, as $\| s_i \|_{C^{0}_{\delta}}  \leq \| s_i \|_{C^{2,\alpha}_{\delta}} = 1$, the sequence is also bounded from above in this norm.

If we then rescale the $s_i$ so they satisfy that $\| s_i \|_{C^0_{\delta}} = 1$, these properties then imply that 
\begin{align}\label{c2estimate} \| s_i \|_{C^{2,\alpha}_{\delta}}  < C
\end{align}
and that both $\| \Delta_{\varepsilon_i} (s_i) \|_{C^{0,\alpha}_{\delta - 2}} \rightarrow 0 $ and $s_i \rightarrow 0$ in $C^{2,\alpha}_{\textnormal{loc}} (X_p , \End E)$.

Next, we will make use of the functions $$\varrho = \varrho_{\varepsilon} : \Bl_p X \rightarrow \mathbb{R}$$ given by
\begin{align*} \varrho (x) =
  \begin{cases}
    1     & \quad \text{if } x \in \Bl_p X \setminus \pi^{-1} (B_{1}) \\
    |z(x)|     & \quad \text{if } x \in  \pi^{-1} (B_{1} \setminus B_{\varepsilon}) \\
    \varepsilon      & \quad \text{if } x \in \pi^{-1} (B_{\varepsilon} ) \\
  \end{cases}
\end{align*}
Here we are using the coordinate $z$ on the annular region $\pi^{-1} ( B_1 \setminus B_{\varepsilon})$ as before. For notational simplicity, denote $\varrho_i = \varrho_{\varepsilon_i}$.

Note that $\| s_i \|_{C^0_{\delta}}  = \| \varrho_i^{- \delta} s_i \|_{C^0(\Bl_p X, \End E)}$, so for each $i$ there is a point $q_i \in \Bl_p X$ such that
\begin{align}\label{c0bound} | \varrho_i^{-\delta} (q_i) s_i (q_i) | = 1.
\end{align} 
Since $|s_i(q_i) | \rightarrow 0$, this then implies that $\varrho_i^{-\delta} (q_i) \rightarrow \infty$. Thus, as $\delta < 0$, we also have $\varrho_i (q_i) \rightarrow 0$. We separate the two cases when $\varepsilon_i^{-1} \varrho_i (q_i)$ is bounded and not. As we will see, this corresponds to whether or not the points $q_i$, thought of as points in $\Bl_0 \mathbb{C}^n$ using the $\varepsilon$-dependent charts, remain bounded or not. 

We first consider the case that $\varepsilon_i^{-1} \varrho_i (q_i)$ is bounded. Recall that we have the coordinates $\zeta = \varepsilon^{-1} z$ on the neighbourhood of the exceptional divisor in the blowup. That $\varepsilon_i^{-1} \varrho_i (q_i) \leq r$ then says that $|\zeta(q_i)| \leq r$. Thus we can choose a subsequence such that the points $\zeta(q_i)$ converge to a point $q \in \Bl_0 \mathbb{C}^n$. 

Through our coordinate charts $\zeta$, the $s_i$ can be thought of as sections of $$\mathbb{C}^{m^2} \rightarrow \pi^{-1} \left(B_{\varepsilon^{-1}_i} (0) \right) \subseteq \Bl_0 \mathbb{C}^n.$$ Moreover, since the norm $\| \cdot \|_{C^{2,\alpha}_{\delta}}$ on $\Bl_p X$ for the parameter $\varepsilon$ is equivalent to $\| \gamma_1 \cdot \|_{C^{2,\alpha}_{\delta} (X_p, \End E)} + \varepsilon^{-\delta} \| \gamma_2 \cdot \|_{C^{2,\alpha}_{\delta} (\Bl_0 \mathbb{C}^n, \mathbb{C}^{m^2}) }$, estimate (\ref{c2estimate}) gives a a uniform $C^{2,\alpha}_{\delta}$ bound on $\varepsilon_i^{-\delta} s_i$ increasing subsets of $\Bl_0 \mathbb{C}^n$ whose union covers the whole of $\Bl_0 \mathbb{C}^n$. 

We can then choose a subsequence of $\varepsilon_i^{-\delta} s_i$, which we still denote by the same index, converging in $C^{2,\alpha}_{\textnormal{loc}}$ to a section $s$ of $\mathbb{C}^{m^2} \rightarrow \Bl_0 \mathbb{C}^n$. The bounds on the $\varepsilon_i^{-\delta} s_i$ then imply that $s \in C^{2,\beta}_{\delta} (\Bl_0 \mathbb{C}^n, \mathbb{C}^{m^2})$ for some $\beta < \alpha$.

The pair $(\varepsilon_i^{-2} \omega_{\varepsilon_i}, h_{\varepsilon_i})$ converges to $(\omega_{BS}, h_{\textnormal{flat}})$ on $\Bl_0 \mathbb{C}^n$. Thus as the $s_i$ solve $\Delta_{\varepsilon_i} (s_i) =0$,  the limit $s$ is a solution of $\Delta(s) = 0$, where $\Delta$ is the Laplace operator associated to $(\omega_{BS}, h_{\textnormal{flat}})$. Since $\varepsilon_i^{-1} \varrho_i (q_i) \leq r$ and $\delta < 0$, we have $ \varrho_i^{\delta} (q_i)  \geq r^{\delta}$. Therefore, by equation (\ref{c0bound}), $|s(q)| \geq r^{\delta}$. But on $C^{2,\beta}_{\delta} (\Bl_0 \mathbb{C}^n, \mathbb{C}^{m^2})$, Corollary \ref{relevantwtsblowupcn} gives that this Laplacian has trivial kernel, and so this is a contradiction.  

To complete the proof we must show that the case when $\varepsilon_i^{-1} \varrho_i (q_i)$ is unbounded also leads to a contradiction. Recall that $\varrho_i (q_i) \rightarrow 0$, so the only possibility then is that infinitely many $q_i$ lie in $\pi^{-1} ( B_1 \setminus B_{\varepsilon_i} )$. Thus we can choose a subsequence, which we still denote by the same index, such that $\varrho_i(q_i) = |z(q_i)|$ for all $i$. 

Now choose $0 < r_i < R_i$ such that $\varepsilon_i^{-1} r_i |z(q_i)| \rightarrow \infty$ and $R_i |z(q_i)| \rightarrow 0$. By our assumptions, it then follows that $r_i \rightarrow 0$ and that $R_i \rightarrow \infty$. We will identify the annular region 
\begin{align*} \{ x \in \Bl_p X : r_i |z(q_i) | \leq |z(x)| \leq R_i |z(q_i) | \}
\end{align*}
with the annulus
\begin{align*} \mathfrak{A}_i =  B_{R_i} (0) \setminus B_{r_i} (0)  \subseteq \mathbb{C}^n \setminus \{ 0 \}
\end{align*} by rescaling. Note that the $q_i$, thought of as points in the latter annulus, then all lie on the unit sphere in $\mathbb{C}^n \setminus \{ 0 \}$. So we can choose a convergent subsequence $q_i \rightarrow q \in S^{2n-1} \subseteq \mathbb{C}^n \setminus \{ 0 \}$.

Next, the metrics $( \frac{1}{|z(q_i)|^2} \omega_{\varepsilon_i}, h_{\varepsilon_i})$ on $\End E \cong \mathbb{C}^{m^2}$ over $\mathfrak{A}_i$ converge to $(\omega_{\textnormal{eucl}}, h_{\textnormal{flat}})$ locally uniformly on $\mathbb{C}^n \setminus \{ 0 \}$ to any order. This follows by our choice of the $r_i$ and $R_i$. The latter ensures that the parts coming from $\omega$ on $X$ approach the Euclidean metric, as $\omega$ is approximately Euclidean near $p$ and the balls $B_{R_i |q_i|}$ become smaller and smaller balls around $p$. The former ensures the same from the terms coming from the asymptotically flat Burns-Simanca metric, as the annular region thought of as a subset of $\Bl_0 \mathbb{C}^n$ lies in the complement of $\pi^{-1} (B_{\varepsilon_i^{-1} r_i |z(q_i)})$, which becomes a larger and larger subset of $\Bl_0 \mathbb{C}^n$ as $i$ increases.

The Laplace operators corresponding to the weight $\varepsilon_i$ therefore converge to the Laplace operator associated to $(\omega_{\textnormal{eucl}}, h_{\textnormal{flat}})$ on $\mathbb{C}^{m^2} \rightarrow \mathbb{C}^n \setminus \{ 0 \}$. 

Also, the sections $|z(q_i)|^{-1} s_i$ on $\mathfrak{A}_i$ have $C^{2,\alpha}_{\delta} (\mathfrak{A}_i, \mathbb{C}^{m^2})$-norm which is uniformly equivalent to the $C^{2,\alpha}_{\delta} (B_{|z(q_i)| R_i)} \setminus B_{|z(q_i)| r_i)}, \End E)$-norm of the $s_i$. Since the $s_i$ satisfy the estimate (\ref{c2estimate}), this therefore implies that the sections $|z(q_i)|^{-1} s_i$ satisfy the same estimate on $\mathfrak{A}_i$ as well. We can then extract a subsequence of $|z(q_i)|^{-1} s_i$ converging to a section $s$ of $\mathbb{C}^{m^2}$ over $\mathbb{C}^n \setminus \{ 0 \}$ locally in $C^{2,\beta}$ for some $\beta < \alpha$. The uniform bound on the $C^{2,\alpha}_{\delta} (\mathfrak{A}_i, \mathbb{C}^{m^2})$-norm of $|z(q_i)|^{-1} s_i$ then implies that $s$ has finite $C^{2,\beta}_{\delta}$-norm.

Thus we get a function $s: \mathbb{C}^n \setminus \{ 0 \} \rightarrow \mathbb{C}^{m^2}$ in $C^{2,\beta}_{\delta} (\mathbb{C}^n \setminus \{ 0 \} , \mathbb{C}^{m^2} )$ which satisfies $\Delta (s) = 0$, where $\Delta$ is the Euclidean Laplacian acting on each component separately. Moreover, since $\varrho_i (q_i) = |z(q_i)|$ for all $i$,  equation (\ref{c0bound}) implies that $|s(q) | = 1$. But for any weight which is not an indicial root, this Laplacian is an isomorphism. Thus $\Delta (s) =0$ implies that $s =0$ and so $s(q) = 0 \neq 1$. Hence no such $s$ can exist. This means that we cannot have that $\varepsilon^{-1}_{i} \varrho_i (q_i)$ is unbounded either, and this completes the proof of the result.
\end{proof}

\section{Proof of the main result}\label{main-section}

Having developed the linear theory, we are now ready to prove the main theorem. The strategy of the proof is analogous to Sz\'ekelyhidi's method in the setting of extremal K\"ahler metrics \cite{szekelyhidi12}, see also \cite[Chapter 8]{szekelyhidi14book}. To ease notation, in this section we will often denote $C^{k,\alpha}_{\delta} (\Bl_p X,  \End E) $ simply by $C^{k,\alpha}_{\delta}$, where as usual we are taking $\delta \in (2-2n,0)$.

We briefly recall the setup of our problem. On $\pi^*E\to \Bl_p X$ we have a family of hermitian metrics $h_{\epsilon}$ and also a family of K\"ahler metrics $\omega_{\epsilon}$. On $(E,h)\to (X,\omega)$ we have a canonical connection $A$ - the Hermitian Yang-Mills connection - which induces a holomorphic structure on $E$. This induces a holomorphic structure on $\pi^*E$, hence we obtain a Chern connection for this holomorphic structure with respect to each $h_{\epsilon}$. We will denote this connection by $A_{\epsilon}$. For each element of the complex gauge group $f\in  \Gamma (\GL (E))$, we obtain a new connection $A_{\epsilon}^f$ as usual by setting $$d_{A_{\epsilon}^f} = f^* \circ \partial_{A_{\epsilon}} \circ (f^*)^{-1} + f^{-1} \circ \bar{\partial}_{A_{\epsilon}} \circ f.$$ As in equation (\ref{linearisation-space}), motivated by the linearisation calculation, we shall set $$ \scH_{\epsilon}= \mathcal{G}^{\mathbb{C}} \cap \Gamma \left( \End_{H} ( E, h_{\varepsilon}) \right),$$ where $\End_{H} (E, h_{\varepsilon})$ denotes the space of Hermitian endomorphisms of $(E,h_{\epsilon})$. With this notation in place, the equation we wish to solve is
\begin{align*} i \Lambda_{\omega_{\varepsilon}} F_{A_{\epsilon}^f} = c_{\varepsilon} \cdot \Id_E,
\end{align*}
where $c_{\epsilon}$ is the only possible topological constant and $f\in  \scH_{\epsilon}$.

We wish to solve the equation using the contraction mapping theorem by rephrasing the equation above as a fixed point problem. For this to have any hope, we need solutions of the equation to be unique. As it stands, this is not the case: if $f \in \mathcal{G}^{\mathbb{C}}$ is a solution, then so is $c f $, for any non-zero constant $c$. We avoid this issue by modifying the equation we wish to solve. As in the previous section, pick a point $q \in \Bl_p X$ sufficiently far away from the exceptional divisor. We shall instead try to solve the equation
$$i \Lambda_{\omega_{\varepsilon}} F_{A_{\varepsilon}^f}  = c_0\cdot \Id_E + \tr_q (f) \cdot \Id_{E},$$ which has no such issue:  if a solution exists, we are forcing $\tr_q(f) = m(c_{\epsilon} - c_0)$. 

 To see how we can view this equation as a fixed point problem, we expand the operator $F_{A_{\epsilon}^f}$ as 
\begin{align*} i \Lambda_{\omega_{\varepsilon}} F_{A_{\epsilon}^f} = i \Lambda_{\omega_{\varepsilon}} F_{A_{\varepsilon}} + L_{\varepsilon} (f) + Q_{\varepsilon} (f)
\end{align*}
for some non-linear operator $Q_{\varepsilon} (f)$. Here, since we are restricting to Hermitian endomorphisms, the linearisation  $L_{\varepsilon} (f) = \Delta_{A_{\epsilon,\End E}}(f)$ is the Laplacian, by Lemma \ref{linearisedoperator}. 
The equation can then be viewed as
\begin{align}\label{eqn1} L_{\varepsilon} (f) - \tr_q (f) \cdot \Id_E =  c_0 \cdot \Id_E - i \Lambda_{\omega_{\varepsilon}} F_{A_{\varepsilon}}  -  Q_{\varepsilon} (f).
\end{align} The left hand side of equation (\ref{eqn1}) is nothing but the operator $\tilde{L}_{\varepsilon}$ as in Section \ref{blowuplinear}. By Theorem \ref{blowuplinearthm}, this operator has an inverse $\tilde{L}_{\varepsilon}^{-1}$ and so the equation we wish to solve is 
\begin{align*} f = \mathcal{N}_{\varepsilon} (f),
\end{align*}
where $$\mathcal{N}_{\varepsilon} : C^{2,\alpha}_{\delta} (\Bl_p X,  \End_{H} ( E, h_{\varepsilon})) \rightarrow C^{2,\alpha}_{\delta } (\Bl_p X,  \End_{H} ( E, h_{\varepsilon}) )$$ is the operator 
\begin{align}\label{mainoperator} \mathcal{N}_{\varepsilon} (f) = \tilde{L}_{\varepsilon}^{-1} \left( c_0 \cdot \Id_E -i \Lambda_{\omega_{\varepsilon}} F_{A_{\varepsilon}}  -  Q_{\varepsilon} (f) \right) .
\end{align}

We now prove a series of lemmata that will contain the key estimates for the proof of Theorem \ref{blthm}. We start by giving an estimate on the annular region for the connection $A_{\varepsilon}$.

\begin{lem}\label{chernconnbound} With $A_{\varepsilon}$ as above, in the annular region $B_{2 r_{\varepsilon}} \setminus B_{r_{\varepsilon}}$ there is a $C>0$, independently of $\varepsilon$, such that 
\begin{align*} \| A_{\varepsilon} -  A \|_{C^{k,\alpha}_{2}} \leq C.  \end{align*}

\end{lem} 
\begin{proof} To prove this estimate, we use the local formula for $A_{\varepsilon}$ as $d + \mathcal{A}_{\varepsilon}$, where $\mathcal{A}_{\varepsilon}$ is the matrix-valued $(1,0)$-form given by
\begin{align}\label{localchernconn} \mathcal{A}_{\varepsilon} = \bar{h}_{\varepsilon}^{-1} \partial \left( \bar{h}_{\varepsilon} \right)\end{align} and we are thinking of $h_{\varepsilon}$ as a matrix. 

In the annulus $B_{2 r_{\varepsilon}} \setminus B_{r_{\varepsilon}}$, $h_{\varepsilon}$ is given by
\begin{align*}h_{\varepsilon} = \gamma_1 h + \gamma_2 h_{\textnormal{flat}}.
\end{align*} By \cite[Lemma 5.2]{XW}, we can choose a local frame for $E$ such that $h = h_{\textnormal{flat}}$ up to an element of $C^{2,\alpha}_{2}$, so we may may replace $h$ with $h_{\textnormal{flat}}$ in our estimate.  

By definition $\gamma_1 + \gamma_2 = 1$, and equation $(\ref{cutoffbound})$ gives that $\| \gamma_i \|_{C^{2,\alpha}_{0}} \leq c$. Thus the multiplicative properties of the weighted norms proved in Lemma \ref{productembedding} imply that $h_{\varepsilon}$ also agrees with $h_{\textnormal{flat}}$ up to an element of $C^{2,\alpha}_{2}$, and hence so does $h_{\varepsilon}^{-1}$. Thus we may take $h_{\varepsilon}^{-1}$ to be the identity in the estimation of equation (\ref{localchernconn}). 

Next, by the Leibniz rule we have
\begin{align*}\partial \bar h_{\varepsilon} &= \gamma_1 \partial \bar  h + (\partial\gamma_1) \bar h + (\partial\gamma_2) \bar h_{\textnormal{flat}} + \gamma_2 \partial  \bar h_{\textnormal{flat}} \\ 
&=  \gamma_1 \partial \bar h + (\partial\gamma_1) \bar h + (\partial\gamma_2) \bar h_{\textnormal{flat}} .
\end{align*}
But $\gamma_2 = 1 - \gamma_1$, so 
\begin{align*}(\partial\gamma_1) \bar h + (\partial\gamma_2) \bar h_{\textnormal{flat}} &= (\partial\gamma_1) \left( \bar h - \bar h_{\textnormal{flat}}  \right).
\end{align*}
Again $\bar h - \bar h_{\textnormal{flat}}  \in C^{2,\alpha}_2$, hence the multiplicative properties of the weighted norms imply that 
\begin{align*}(\partial\gamma_1) \bar h + (\partial\gamma_2) \bar h_{\textnormal{flat}} \in C^{2,\alpha}_2.
\end{align*}
Moreover, the norm of this element can be bounded above independently of $\varepsilon$, because of the multiplicative properties and the bound on the norm of the $\gamma_i$ in equation \eqref{cutoffbound}. This gives the desired result.
\end{proof}

Next we use this to show that the approximate solution is small. 

\begin{lem}\label{approxsol}Suppose $\delta \in (2-2n, 0)$. Then there are constants $C, \varepsilon_0 >0$ such that for all $\varepsilon \in (0,\varepsilon_0)$, 
\begin{align*} \| \mathcal{N} (\Id_E) \|_{C^{2,\alpha}_{\delta} (\Bl_p X, \End E)} \leq C r_{\varepsilon}^{2-\delta}.
\end{align*}
\end{lem}
\begin{proof} The quantity we wish to estimate is $$\mathcal{N}_{\epsilon}(\Id_E) =   \tilde{L}_{\varepsilon}^{-1} \left( c_0 \cdot \Id_E -\Lambda_{\omega_{\varepsilon}} F_{A_{\varepsilon}}\right),$$ since $Q_{\epsilon}(\Id_E)$ is zero. Note first that by Theorem \ref{blowuplinearthm}, $\tilde{L}_{\varepsilon}^{-1}$ is bounded independently of $\varepsilon$, thus it suffices to establish the bound
\begin{align*} \| c_0  \cdot \Id_E - i \Lambda_{\omega_{\varepsilon}} F_{A_{\varepsilon}} \|_{C^{0,\alpha}_{\delta - 2} } \leq C r_{\varepsilon}^{2-\delta}
\end{align*}
for some $C > 0$. Also, on $\Bl_p X \setminus \pi^{-1} \left( B_{2 r_{\varepsilon} } \right)$, $\Lambda_{\omega_{\varepsilon}} F_{A_{\varepsilon}} = c_0 \cdot \Id_E$, so we have to prove the above bound in the region $\pi^{-1} (B_{2r_{\varepsilon}})$ near the exceptional divisor. We consider the terms $c_0 \cdot \Id_E$ and $\Lambda_{\omega_{\varepsilon}} F_{A_{\varepsilon}}$ separately. 

For $c_0 \cdot \Id_E$, we need to estimate 
\begin{align*}  \sup_{r\in (\varepsilon, 2r_{\varepsilon})} \| (\Id_E)_r^{\delta - 2} \|_{C^{k,\alpha} (B_2 \setminus B_1)} + \| (\Id_E)^{\delta - 2}_{\varepsilon} \|_{\Bl_0 \mathbb{C}^n}.
\end{align*}
Remark that $\delta \in (2-2n,0)$, so in particular $2-\delta>0$. As $ (\Id_E)_r^{\delta} = r^{2-\delta} \Id_E$, the supremum in the first term above is attained at the largest value of $r$, which is $2r_{\varepsilon}$. This gives that 
\begin{align*}  \sup_{r\in (\varepsilon, 2r_{\varepsilon})} \| (\Id_E)_r^{\delta - 2} \|_{C^{k,\alpha} (B_2 \setminus B_1)} = c_1 r_{\varepsilon}^{2-\delta}
\end{align*}
for some fixed constant $c_1$ independent of $\epsilon$. Similarly, 
\begin{align*}   \| (\Id_E)^{\delta - 2}_{\varepsilon} \|_{\Bl_0 \mathbb{C}^n} = \varepsilon^{2-\delta} c_2
\end{align*}
for some fixed constant $c_2$ independent of $\epsilon$. 

By definition $r_{\varepsilon}^{2-\delta} = \varepsilon^{2- \delta} \varepsilon^{(1-\kappa) (\delta - 2) }$. Together with $\delta - 2 < 0$ and $1-\kappa > 0$, this gives that $\varepsilon^{2-\delta} < r_{\varepsilon}^{2 - \delta}$, implying the required bound on  $\Id_E$.

Finally, we must show that $\Lambda_{\varepsilon} F_{A_{\varepsilon}}$ satisfies the desired bound and it is here that we apply Lemma \ref{chernconnbound}. Noting that $F_{A_{\varepsilon}}$ vanishes inside $\pi^{-1} (B_{r_{\varepsilon}})$ and $r_{\varepsilon} > \varepsilon$, we need to estimate
\begin{align*}\sup_{r\in (r_{\varepsilon}, 2r_{\varepsilon})} \| (\Lambda_{\omega_{\varepsilon}} F_{A_{\varepsilon}})_r^{\delta-2} \|_{C^{0,\alpha} (B_2 \setminus B_1)}  = \| \Lambda_{\omega_{\varepsilon}} F_{A_{\varepsilon}} \|_{C^{0,\alpha}_{\delta-2} (B_{2r_{\varepsilon}} \setminus B_{r_{\varepsilon}} ) }.
\end{align*}
In other words, we must show there is a constant $c>0$ such that
\begin{align}\label{ineqtoshow} \| \Lambda_{\omega_{\varepsilon}} F_{A_{\varepsilon}} \|_{C^{0,\alpha}_{\delta-2} (B_{2r_{\varepsilon}} \setminus B_{r_{\varepsilon}} ) } \leq c r_{\varepsilon}^{2-\delta}.
\end{align}

By Lemma \ref{chernconnbound}, the difference between the curvatures $F_{A_{\varepsilon}}$ and $F_{A}$ on the annular region is $O(|z|^2)$, independently of $\varepsilon$, where as above $A$ is the Chern connection of the Hermitian-Einstein metric on $E \rightarrow X$. Thus the multiplicative properties imply that there is a constant $c_1 > 0$ such that 
\begin{align}\label{ineq1} \| \Lambda_{\omega_{\varepsilon}} F_{A_{\varepsilon}} \|_{C^{0,\alpha}_{\delta-2}  (B_{2r_{\varepsilon}} \setminus B_{r_{\varepsilon}} ) }  \leq c_1 \left( \| \Lambda_{\omega_{\varepsilon}} F_{A} \|_{C^{0,\alpha}_{\delta-2} (B_{2r_{\varepsilon}}  \setminus B_{r_{\varepsilon}}  )} +  r_{\varepsilon}^{2} r_{\varepsilon}^{2- \delta} \right).
\end{align}
The factor $r_{\varepsilon}^{2} r_{\varepsilon}^{2- \delta}$ being, up to a constant multiple, the norm of $|z|^2$ in the annular region $B_{2r_{\varepsilon}} \setminus B_{r_{\varepsilon}}$. Since $r_{\varepsilon} \rightarrow 0$, this factor is bounded above by $r_{\varepsilon}^{2 - \delta}$. So to establish the desired inequality, all that remains is to bound $\| \Lambda_{\omega_{\varepsilon}} F_{A} \|_{C^{0,\alpha}_{\delta-2} (B_{2r_{\varepsilon}}  \setminus B_{r_{\varepsilon}}  )}$.

 On the annular region under consideration, setting $\omega_{\epsilon} = \omega + i\partial \bar{\partial} \phi_{\epsilon}$, equation (\ref{potentialbound}) implies the bound $\| \phi_{\epsilon} \|_{C^{4,\alpha}_{\delta}} \leq c_2 r_{\varepsilon}^{4-\delta}$ for some constant $c_2$. Therefore by \cite[Lemmata 8.13 and 8.19]{szekelyhidi14book}, the contraction operators $\Lambda_{\omega_{\varepsilon}}$ and $\Lambda_{\omega}$ differ by $c_3 r_{\varepsilon}^{4-\delta}$ in the $\delta-2$ weighted operator norm on $B_{2r_{\varepsilon}} \setminus B_{r_{\varepsilon}}$. Applying this to $F_{A}$ and using the multiplicative properties of the norm, we get that
\begin{align*} \| \Lambda_{\omega_{\varepsilon}} F_{A} - \Lambda_{\omega} F_{A} \|_{C^{0,\alpha}_{\delta-2} (B_{2r_{\varepsilon}} \setminus B_{r_{\varepsilon}} ) } & \leq c_3 r^{4-\delta} \| F_{A} \|_{C^{0,\alpha}_0 (B_{2r_{\varepsilon}} \setminus B_{r_{\varepsilon}} ) } \\
& \leq c_4 r_{\varepsilon}^{4 - \delta}
\end{align*}
for some constant $c_4 > 0$. To establish the desired inequality (\ref{ineqtoshow}), by the above and (\ref{ineq1}), it suffices to show that a similar inequality holds for 
\begin{align*} \| \Lambda_{\omega} F_{A} \|_{C^{0,\alpha}_{\delta -2} (B_{2r_{\varepsilon}} \setminus B_{r_{\varepsilon}}) }.
\end{align*}
But since $i \Lambda_{\omega} F_{A} = c_0 \cdot \Id_E$, this reduces to the first case treated above. This completes the proof.
\end{proof}

Next we show that a small change in the connection gives a small change in the linearised operator. For $f \in \scH_{\epsilon}$, we will let $L_{\varepsilon,f}$ denote the linearised operator associated to the connection $A_{\varepsilon}^{f}$ instead of $A_{\varepsilon}$, so that $L_{\varepsilon, \Id_E} = L_{\varepsilon}$. We will also use the $C^{2,\alpha}_{\delta} \rightarrow C^{0,\alpha}_{\delta -2 }$-norm, which is simply the operator norm between the two indicated weighted spaces. 

\begin{lem}\label{linoppert} There exist  $c,C>0$ such that if $f \in\scH_{\epsilon} $ satisfies  $\| f - \Id_E \|_{C^{2,\alpha}_{0}} \leq c$, then 
\begin{align*} \| L_{\varepsilon, f} - L_{\varepsilon} \|_{C^{2,\alpha}_{\delta} \rightarrow C^{0,\alpha}_{\delta -2 } } \leq C \| f - \Id_E \| _{C^{2,\alpha}_{0}} .
\end{align*}
\end{lem}
\begin{proof} For $s \in C^{2,\alpha}_{\delta}(\End E$), by definition of the Laplacian $\Delta_{\epsilon,\End E}$ we need to estimate
\begin{align*} \left( \partial_{A^{f}_{\varepsilon}} \bar{\partial}_{A_{\varepsilon}^{f}} - \bar{\partial}_{A_{\varepsilon}^{f}} \partial_{A^{f}_{\varepsilon}} - \partial_{A_{\varepsilon}} \bar{\partial}_{A_{\varepsilon}} + \bar{\partial}_{A_{\varepsilon}} \partial_{A_{\varepsilon}}\right) (s).
\end{align*}
We only present the argument for 
\begin{align*}  \left( \partial_{A^{f}_{\varepsilon}} \bar{\partial}_{A_{\varepsilon}^{f}} - \partial_{A_{\varepsilon}} \bar{\partial}_{A_{\varepsilon}} \right) (s),
\end{align*}
as the argument for the remaining terms is similar.

Let $\widetilde{f}$ denote $f - \Id_E$ and let $v$ denote $f^{-1} - \Id_E$. Recalling the definition of the action of $f$ on $d_{A_{\varepsilon}}$ in equation (\ref{connectionaction}) and using that $f^* = f$ for $f \in \mathcal{H}_{\varepsilon}$ 
\begin{align*}\partial_{A^{f}_{\varepsilon}} &=  f \circ \partial_{A_{\varepsilon}} \circ f^{-1} , \\
\bar{\partial}_{A^{f}_{\varepsilon}} &=  f^{-1} \circ \bar{\partial}_{A_{\varepsilon}} \circ f .
\end{align*}
Thus, noting that $f^{-1} \circ f^{-1} = ( \Id_E + v) \circ (\Id_E + v) = \Id_E + 2 v + v \circ v$, we see
\begin{align*} \partial_{A^{f}_{\varepsilon}} \bar{\partial}_{A_{\varepsilon}^{f}} =& f \circ \partial_{A_{\varepsilon}} \circ f^{-1} \circ f^{-1} \circ  \bar{\partial}_{A_{\varepsilon}} \circ f \\ =& \widetilde{f} \circ \partial_{A_{\varepsilon}} \circ \bar{\partial}_{A_{\varepsilon}} \circ \widetilde{f}  + 2 \widetilde{f} \circ \partial_{A_{\varepsilon}} \circ  v  \circ  \bar{\partial}_{A_{\varepsilon}} \circ \widetilde{f}  + \widetilde{f} \circ \partial_{A_{\varepsilon}} \circ v \circ v  \circ  \bar{\partial}_{A_{\varepsilon}} \circ \widetilde{f} \\
& +  \widetilde{f} \circ \partial_{A_{\varepsilon}} \circ \bar{\partial}_{A_{\varepsilon}}  +  2 \widetilde{f} \circ \partial_{A_{\varepsilon}} \circ  v \circ  \bar{\partial}_{A_{\varepsilon}} +  \widetilde{f} \circ \partial_{A_{\varepsilon}} \circ  v \circ v  \circ  \bar{\partial}_{A_{\varepsilon}} \\
&+  \partial_{A_{\varepsilon}} \circ \bar{\partial}_{A_{\varepsilon}} \circ \widetilde{f} +  2 \partial_{A_{\varepsilon}} \circ v  \circ  \bar{\partial}_{A_{\varepsilon}} \circ \widetilde{f} +  \partial_{A_{\varepsilon}} \circ  v \circ v \circ  \bar{\partial}_{A_{\varepsilon}} \circ \widetilde{f} \\
&+ \partial_{A_{\varepsilon}} \circ \bar{\partial}_{A_{\varepsilon}} + 2 \partial_{A_{\varepsilon}} \circ v \circ  \bar{\partial}_{A_{\varepsilon}} + \partial_{A_{\varepsilon}} \circ v \circ v \circ  \bar{\partial}_{A_{\varepsilon}} .
\end{align*}
The important point to note is that if we choose $\widetilde{f} = f - \Id_E$ sufficiently small in the $C^{2,\alpha}_{0}$-norm, then there is a $C_1>0$ such that $\| v \|_{C^{2,\alpha}_{0}} \leq C_1 \| \widetilde{f} \|_{C^{2,\alpha}_{0}}$. Using the multiplicative properties of the norm, all of the terms above except $\partial_{A_{\varepsilon}} \circ \bar{\partial}_{A_{\varepsilon}} $ can then be bounded by some constant multiple of $$ \| \widetilde{f} \|_{C^{2,\alpha}_{0}} \cdot \| \partial_{A_{\varepsilon}} \circ \bar{\partial}_{A_{\varepsilon}} \|_{C^{2,\alpha}_{\delta} \rightarrow C^{0,\alpha}_{\delta - 2}}.$$ But since $\partial_{A_{\varepsilon}} \circ \bar{\partial}_{A_{\varepsilon}}$ is a bounded operator, by possibly increasing the constant, we can bound this term by a constant times $ \| \widetilde{f} \|_{C^{2,\alpha}_{0}} $. It follows that the same holds for the induced connection on endomorphisms of $E$, and this completes the proof.
\end{proof}

This in turn allows us to show that the non-linear operator $\mathcal{N}_{\varepsilon}$ is a contraction with a specific constant in a ball around $\Id_E$.
\begin{lem}\label{Ncontraction} There exists a $c>0$ such that if $f, f' \in\scH_{\epsilon} $ satisfy 
\begin{align*} \| f - \Id_E \|_{C^{2,\alpha}_{0}},\| f' - \Id_E \|_{C^{2,\alpha}_{0}} \leq c
\end{align*} 
then 
\begin{align*} \| \mathcal{N}_{\varepsilon} (f ) - \mathcal{N}_{\varepsilon} (f') \|_{C^{2,\alpha}_{\delta}} \leq \frac{1}{2} \| f - f' \|_{C^{2,\alpha}_{\delta}}.
\end{align*}
\end{lem}
\begin{proof} First note that 
\begin{align*} \mathcal{N}_{\varepsilon} (f ) - \mathcal{N}_{\varepsilon} (f') =  \tilde{L}_{\varepsilon}^{-1} \left( Q_{\varepsilon} (f' ) - Q_{\varepsilon} (f) \right).
\end{align*}
By the boundedness of $\tilde{L}_{\varepsilon}^{-1}$, it suffices to show that we can make the norm of 
\begin{align*} Q_{\varepsilon} (f' ) - Q_{\varepsilon} (f)  
\end{align*}
as small a multiple of $\| f - f' \|_{C^{2,\alpha}_{\delta}}$ as we like if we make $f$ and $f'$ sufficiently close to $\Id_E$.

The mean value theorem implies that we can find a $t\in [0,1]$ such that  $s = t f + (1-t) f'$ satisfies
\begin{align*} Q_{\varepsilon} (f' ) - Q_{\varepsilon} (f) &= DQ_{\varepsilon, s} (f') -  DQ_{\varepsilon, s} (f ) \\
&=  DQ_{\varepsilon, s} (f' - f) \\
&= \left( L_{\varepsilon, s} - L_{\varepsilon} \right) (f' - f).
\end{align*}
Here we are using the $s \in \Gamma \left ( \End_H (E , h_{\varepsilon} ) \right)$ so that the linearised operator really is the Laplacian.

By Lemma \ref{linoppert}, if the $C^{2,\alpha}_{0}$-norms of $f- \Id_E$ and $f'- \Id_E$ are sufficently small, we have 
\begin{align*}  \| L_{\varepsilon, s} - L_{\varepsilon} \|_{C^{2,\alpha}_{\delta - 2}} & \leq C \| s - \Id_E \|_{C^{2,\alpha}_{0}} \\ 
& \leq C \left( t \| f - \Id_E \|_{C^{2,\alpha}_{0}} + (1-t) \|f'- \Id_E \|_{C^{2,\alpha}_{0}} \right) \\ 
& \leq C \left( \| f - \Id_E \|_{C^{2,\alpha}_{0}} + \| f' - \Id_E \|_{C^{2,\alpha}_{0}} \right) 
\end{align*}
So by possibly requiring that the $C^{2,\alpha}_{0}$-norms of $f- \Id_E$ and $f'- \Id_E$ are even smaller, we can make this constant as small as we like. This provides the required bound for $\mathcal{N}_{\varepsilon}$, as required.
\end{proof}

We now have all the components required to finish off the proof of Theorem \ref{blthm}. We want to obtain a solution by applying the contraction mapping theorem, so we need to show that there is a choice of $\delta$ such that for all $\varepsilon >0$ sufficiently small, there is a neighbourhood of $\Id_E$ in $C^{2,\alpha}_{\delta} \left(\Bl_p X,  \End_H (E,h)  \right)$ to which the contraction mapping theorem applies. This is the content of the Proposition below.
\begin{prop}\label{finalprop} Let $c$ be the constant in Lemma \ref{Ncontraction}. Let $V_{\varepsilon}$ be defined by
\begin{align*} V_{\varepsilon} = \{ s \in C^{2,\alpha}_{\delta} ( \Bl_p X,  \scH_{\epsilon} ) : \| s- \Id_E \|_{C^{2,\alpha}_{\delta}} \leq c \varepsilon^{- \delta} \} .
\end{align*}
If $\delta \in (2-2n, 0)$ is chosen sufficiently close to $0$, then there exists a $\varepsilon_0 >0$ such that for all $\varepsilon \in (0, \varepsilon_0)$, the map $\mathcal{N}_{\varepsilon}$ restricted to $V_{\varepsilon}$ is a contraction and $\mathcal{N}_{\varepsilon}$ maps $V_{\varepsilon}$ into $V_{\varepsilon}$.
\end{prop}
\begin{proof} First note that as $ \delta < 0$, the relationship between the different weighted norms in equation \eqref{weightcomparison} implies that for $s \in V_{\varepsilon}$, 
\begin{align*} \| s - \Id_E \|_{C^{2,\alpha}_{0}} &\leq \varepsilon^{\delta } \| s- \Id_E \|_{C^{2,\alpha}_{\delta}} \\
& \leq c.
\end{align*}
and so Lemma \ref{Ncontraction} applies. In particular, $\mathcal{N}_{\varepsilon}$ is a contraction on $V_{\varepsilon}$. 

By Lemmata \ref{Ncontraction} and \ref{approxsol}, there is a $C>0$ such that
\begin{align*} \| \mathcal{N}_{\varepsilon} ( s )  \|_{C^{2,\alpha}_{\delta}} & \leq \| \mathcal{N}_{\varepsilon} ( s ) - \mathcal{N}_{\varepsilon} (\Id_E)  \|_{C^{2,\alpha}_{\delta}} + \| \mathcal{N}_{\varepsilon} ( \Id_E)  \|_{C^{2,\alpha}_{\delta}} \\
& \leq \frac{1}{2} \| s - \Id_E\|_{C^{2,\alpha}_{\delta}} + C r_{\varepsilon}^{2-\delta}  \\
& \leq \frac{c }{2} \varepsilon^{ - \delta} +  C r_{\varepsilon}^{2-\delta}.
\end{align*}
To show that $\mathcal{N}_{\varepsilon} (s) \in V_{\varepsilon}$ if $\varepsilon$ is sufficiently small, it therefore suffices to establish the bound 
\begin{align*} C r_{\varepsilon}^{2-\delta} \leq \frac{c}{2} \varepsilon^{ - \delta}.
\end{align*}
So it suffices to show that 
\begin{align*}r_{\varepsilon}^{2-\delta} = \varepsilon^{\delta'} \varepsilon^{ - \delta}
\end{align*}
for some $\delta' >0$, i.e. that 
\begin{align*} (2-\delta) \frac{n-1}{n} >  - \delta.
\end{align*}
This is equivalent to
\begin{align*}\delta >   2(1-n),
\end{align*} 
which can clearly be achieved while still having $\delta <0$, since $n>1$. This completes the proof. 
\end{proof}

\bibliography{biblibrary}
\bibliographystyle{amsplain}

\end{document}